\documentclass[12pt,a4paper]{amsart}
\usepackage{enumerate}
\usepackage{amsmath,amssymb,latexsym}
\def\mathbbm{\mathbf}

\def\er{\mathbb R}
\def\har{{\mathbb H^1}}
\def\a{{\rm{\bf a}}}

\newcommand{\twoline}[2]{\genfrac{}{}{0pt}{}{#1}{#2}}
\newtheorem{theorem}{Theorem}
\newtheorem{remark}{Remark}
\newtheorem{lemma}[theorem]{Lemma}
\newtheorem*{lemma*}{Lemma}
\newtheorem{corollary}[theorem]{Corollary}
\newtheorem*{corollary*}{Corollary}
\newtheorem{definition}[theorem]{Definition}
\newtheorem*{definition*}{Definition}
\newtheorem*{theorem*}{Theorem}

\author[M. Paluszynski]{Maciej Paluszynski}

\author[J. Zienkiewicz]{Jacek Zienkiewicz}

\title[ ]{Dimension independent atomic decomposition for dyadic martingale $\har$}
\begin{document}
\begin{abstract}
We introduce atoms for dyadic atomic $\mathbb H^1$ for which the equivalence between atomic and maximal function definitions is dimension independent.
We give the sharp, up to $\log(d)$ factor, estimates for the $\har\rightarrow L^1$ norm estimates
for  the special maximal function.
\end{abstract}
\keywords{Hardy space, atomic decomposition, dyadic martingale}
\subjclass[2010]{42B25, 60G46}
\date{\today}
\maketitle

\baselineskip=18pt

We define a martingale $\har$ space on $\er^d$
\begin{gather}\label{m_h1}
  M^*f=\sup_n|\mathbb E_nf|, \\
  \|f\|_\har=\|M^*f\|_{L^1},\notag
\end{gather}
where $\mathbb E_n$ is the conditional expectation operator associated with the dyadic grid of scale $2^n$.
There are various equivalent definitions of $\har$. In particular,
it has been proved in \cite{D} that an equivalent norm can be defined by
\begin{gather}\label{s_h1}
  S^*f=\big(\sum_n|\mathbb E_nf- \mathbb E_{n+1}f|^2\big)^\frac12, \\
  \|f\|_\har\sim\|S^*f\|_{L^1},\notag
\end{gather}
with the equivalence constants independent of $d$. Similarly as in the Euclidean case, the atomic
decompositions of  martingale $\har$ have been proved based either on the maximal function or the
square function definitions, see \cite{W}. We note that although the atomic norm obtained in
\cite{W} (based on the atomic decomposition) is equivalent to the maximal norm for any single $d$,
the equivalence constants depend on $d$. The aim of this short note is to fine tune the definition
of atoms, for which the equivalence between atomic and maximal function norms is uniform in $d$. By
the results of \cite{D}, the same decomposition works for the square function \eqref{s_h1} norm.

The motivation for our results is their possible applications.  We note that the proposed atomic
decomposition can be used to obtain dimension explicit estimates for various classical operators
acting on martingale $\har$.  In this note we apply Theorem \ref{theorem1} to estimate the
$\har\rightarrow L^1$ norm of special radial  maximal function modeling classical Hardy-Littlewood
maximal oparator. Similar argument works for the heat semigroup, see remark at the end of the
paper. We are going to address further questions, in particular dimension explicit estimates for
classical SIO, in the future, see Remark \ref{rmrk}.

The study of dimension dependence of classical estimates is not new, see papers \cite{NT},
\cite{ST}, which have motivated this current research.

We define atoms.
\begin{definition}\label{Def}
A function $a_Q$ on $\er^d$ is an atom associated with a dyadic cube $Q$ if
\begin{enumerate}[(a)]
\item $\int a_Q=0$, ${\rm supp}\, a_Q\subset Q$,
\item $\|a_Q\|_{L^1}\le 1$,
\item $\|a_Q\|_{L^\infty}\le\frac{2^{d+1}}{|Q|}$, $d$ - dimension,
\item we have a decomposition
\begin{equation*}
  \Big\{x:|a_Q(x)|>\frac{1}{|Q|}\Big\}\subset\bigcup_sQ_s,
\end{equation*}
where $Q_s$ are essentially disjoint dyadic cubes, $Q_s\subset Q$, satisfying the following two conditions:
\begin{itemize}
\item for $Q_s^\#$ being the dyadic parent of $Q_s$ (one scale above)
\begin{equation*}\label{qsets}
  \frac{1}{|Q_s^\#|}\Big|\int_{Q_s^\#}a_Q(x)\,dx\Big|\le\frac{C}{|Q|},
\end{equation*}
where constant $C$ is independent of the dimension $d$,
\item $a_Q$ is constant on each $Q_s$.
\end{itemize}
\end{enumerate}
\end{definition}
\begin{lemma} For an atom $a_Q$ we have
\begin{equation*}
  \|a_Q\|_\har\le C,
\end{equation*}
where the constant $C$ is independent of the dimension $d$.
\end{lemma}
\begin{proof} Pick a dyadic cube $\tilde Q\subset Q$. We need to compute the average of $a_Q$ over $\tilde Q$.
Suppose  $\tilde Q$ is a cube other than any of the $Q_s$'s. Let $\{Q^\#_j\}$ be a family of all
maximal $Q_s^\#\subset \tilde Q$. We denote by $\langle f\rangle_R$ the average of $f$ over a set $R$. Then $\langle a_Q\rangle_{\tilde Q}=\langle a^\#\rangle_{\tilde Q}$, where
$a^\#$ is obtained from $a_Q$ by replacing its value on $Q_j^\#$ by the constant $\langle a_Q\rangle_{Q^\#_j}$.
By  Definition \ref{Def} (d) we have $|a^\#|\le \frac{C}{|Q|}$ ($C$ independent of the dimension). Now suppose  $\tilde Q$ is one of the $Q_s$'s. Then $a_Q$ is constant on $\tilde Q$ and averaging leaves its value unchanged. Hence $M^*a_Q(x)\le \frac C{|Q|}+|a(x)|$ and the desired $L^1$
estimate follows by Definition \ref{Def} (b).
\end{proof}
\begin{remark} If we remove condition \eqref{qsets} from Definition \ref{Def} and assume
the $L^\infty$ estimate (c) on the entire $Q$, the  statement of the above lemma will remain true, but with linear dependence of the implied constant on the dimension $d$.
In order to see this, one has to use $\|M^*\|_{L^p\rightarrow L^p}\le \frac{C}{p-1}$ for
$p=1+\frac1d$.
\end{remark}
We note, that conditions (b) and (c) of the definition of the atom  are suitable for application of the near-$L^1$ estimates.
This approach however, do not seem to lead for the optimal bounds, see Remark \ref{rota}.

\begin{remark} It seems of interest to find the dimensional statement of the result from \cite{PZ}.
\end{remark}

\begin{theorem}[Atomic decomposition]\label{theorem1} For $f\in\har$ there exist a sequence of atoms $\{a_{Q_i}\}$ and a sequence of constants $\{\lambda_i\}$ such that
\begin{equation*}
  f=\sum_i\lambda_ia_{Q_i}\qquad\text{in }\har,
\end{equation*}
and
\begin{equation*}
  \sum_i|\lambda_i|\le\|f\|_\har.
\end{equation*}
\end{theorem}

\begin{remark} A similar atomic decomposition can be obtained in a similar, direct way using the square function definition of $\har$. We omit the details.
\end{remark}

\begin{proof}
Clearly, we can assume that
\begin{equation*}
  \int f=0,\qquad{\rm supp}\,f\subset Q,\quad |Q|=1.
\end{equation*}
We let the family (finite or infinite) $\{Q_{i_1}\}_{i_1}$ consist of the maximal dyadic subcubes of $Q$, for which the average of $f$ is non-zero. By differentiation, a.e outside of the union of $\{Q_{i_1}\}_{i_1}$ we have $f=0$.
We will define inductively consecutive generations of subcubes. The first generation is the family $\{Q_{i_1}\}_{i_1}$. We now construct the second generation of subcubes $\{Q_{i_1,i_2}\}_{i_2}\subset Q_{i_1}$. Recall, that we denote by $\langle f\rangle_Q$ the average of $f$ over $Q$. Let an integer $R(i_1)$ be defined by
\begin{equation*}
  2^{R(i_1)}\le|\langle f\rangle_{Q_{i_1}}|<2^{R(i_1)+1},\quad i_1=1,2,\dots
\end{equation*}
The cubes $Q_{i_1,i_2}$ are the maximal subcubes of $Q_{i_1}$ for which
\begin{equation*}
  |\langle f\rangle_{Q_{i_1,i_2}}|\ge 2^{R(i_1)+2}.
\end{equation*}
In other words $\{Q_{i_1,i_2}\}$ are the moments of the first ``break" through the level $2^{R(i_1)+2}$. We define integers $\alpha(i_1,i_2)$ and $R(i_1,i_2)$ by
\begin{gather*}
  |Q_{i_1,i_2}|=2^{-\alpha(i_1,i_2)}, \\
  2^{R(i_1,i_2)}\le|\langle f\rangle_{Q_{i_1,i_2}}|<2^{R(i_1,i_2)+1}.
\end{gather*}
We iterate the procedure for each of the cubes $Q_{i_1,i_2}$, and we obtain a family of cubes
\begin{equation*}
  \{Q_{i_1,i_2,\dots,i_l}\}_{l=1,2,\dots}
\end{equation*}
We now write the decomposition of $f$ into ``pre-atoms"
\begin{align*}
  f &=-\mathbbm{1}_{Q}\langle f\rangle_{Q}+f\cdot\mathbbm{1}_{Q\setminus\bigcup_{i_1}Q_{i_1}} +\sum_{i_1}\mathbbm{1}_{Q_{i_1}}\langle f\rangle_{Q_{i_1}}+\\
  &\quad+\sum_{i_1}\Big(-\mathbbm{1}_{Q_{i_1}}\langle f\rangle_{Q_{i_1}}+f\cdot\mathbbm{1}_{Q_{i_1}\setminus\bigcup_{i_2}Q_{i_1,i_2}}+\sum_{i_2}\mathbbm{1}_{Q_{i_1,i_2}}
  \langle f\rangle_{Q_{i_1,i_2}}\Big)+ \\
   &\quad+\sum_{i_1,i_2}\Big(-\mathbbm {1}_{Q_{i_1,i_2}}\langle f\rangle_{Q_{i_1,i_2}}+f\cdot\mathbbm{1}_{Q_{i_1,i_2}\setminus\bigcup_{i_3}Q_{i_1,i_2,i_3}}+\\
   &\qquad\qquad+\sum_{i_3}\mathbbm{1}_{Q_{i_1,i_2,i_3}} \langle f\rangle_{Q_{i_1,i_2,i_3}}\Big)+ \\
   &\quad\dots\\
   &\quad+\sum_{i_1,i_2,\dots,i_s}\Big(-\mathbbm {1}_{Q_{i_1,\dots,i_s}}\langle f\rangle_{Q_{i_1,\dots,i_s}}+f\cdot\mathbbm{1}_{Q_{i_1,\dots,i_s}\setminus\bigcup_{i_{s+1}}Q_{i_1,\dots,i_{s+1}}} +\\ &\qquad\qquad+\sum_{i_{s+1}}\mathbbm{1}_{Q_{i_1,\dots,i_{s+1}}} \langle f\rangle_{Q_{i_1,\dots,i_{s+1}}}\Big)+ \\
   &\quad\dots
\end{align*}
We call ``pre-atoms" associated with dyadic cubes $Q_{i_1,\dots,i_s}$ the functions $a_{Q_{i_1,\dots,i_s}}$, which are the normalized elements of the above decomposition
\begin{equation*}
  a_{Q_{i_1,\dots,i_s}}=\frac{\omega_{Q_{i_1,\dots,i_s}}}{\lambda_{Q_{i_1,\dots,i_s}}},
\end{equation*}
where
\begin{align*}
  \omega_{Q_{i_1,\dots,i_s}}&=-\mathbbm {1}_{Q_{i_1,\dots,i_s}}\langle f\rangle_{Q_{i_1,\dots,i_s}}+\\
   &\qquad+f\cdot \mathbbm{1}_{Q_{i_1,\dots,i_s}\setminus\bigcup_{i_{s+1}}Q_{i_1,\dots,i_{s+1}}}+\\
   &\qquad+\sum_{i_{s+1}}\mathbbm{1}_{Q_{i_1,\dots,i_{s+1}}} \langle f\rangle_{Q_{i_1,\dots,i_{s+1}}},
\end{align*}
and
\begin{align*}
  \lambda_{Q_{i_1,\dots,i_s}}&=2^{R(i_1,\dots,i_s)+1}|Q_{i_1,\dots,i_s}|+\\
  &\qquad+2^{R(i_1,\dots,i_s)+2}|Q_{i_1,\dots,i_s}|+\\
  &\qquad+\sum_{i_{s+1}}2^{R(i_1,\dots,i_{s+1})+1}|Q_{i_1,\dots,i_{s+1}}|.
\end{align*}
We include in the above the first ''pre-atom" of the decomposition
\begin{equation*}
  \omega_Q=f\cdot\mathbbm{1}_{Q\setminus\bigcup_{i_1}Q_{i_1}} +\sum_{i_1}\mathbbm{1}_{Q_{i_1}}\langle f\rangle_{Q_{i_1}}, \qquad a_Q=\frac{\omega_Q}{\lambda_Q},
\end{equation*}
where
\begin{equation*}
  \lambda_Q=1+\sum_{i_1}2^{R(i_1)}|Q_{i_1}|.
\end{equation*}
We immediately obtain
\begin{itemize}
\item $\int a_{Q_{i_1,\dots,i_s}}=0$,\ ${\rm supp}\,a_{Q_{i_1,\dots,i_s}}\subset Q_{i_1,\dots,i_s}$,
\item the decomposition
\begin{equation*}
  f=\sum_{s=1}^\infty\sum_{i_1,\dots, i_s}\lambda_{Q_{i_1,\dots,i_s}}\cdot a_{Q_{i_1,\dots,i_s}},\quad\text{in $\har$}.
\end{equation*}
\end{itemize}
Observe, that by the definition of $R(i_1,\dots,i_s)$ we have
\begin{equation*}
    2^{R(i_1,\dots,i_s)}\le|\langle f\rangle_{Q_{i_1,\dots,i_s}}|<2^{R(i_1,\dots,i_s)+1},
\end{equation*}
and similarly for the cubes $Q_{i_1,\dots,i_{s+1}}$ ( $R(i_1,\dots,i_s)$
 replaced by $R(i_1,\dots,i_{s+1})$). On the cube $Q_{i_1,\dots,i_s}$, outside $\bigcup_{i_{s+1}}Q_{i_1,\dots,i_{s+1}}$ we have
\begin{equation*}
  |f|\le 2^{R(i_1,\dots,i_s)+2}
\end{equation*}
(from the definition of $Q_{i_1,\dots,i_{s+1}}$ and differentiation of integrals).
Further, for $x\in Q_{i_1,\dots,i_s}\setminus\bigcup_{i_{s+1}}Q_{i_1,\dots,i_{s+1}}$
\begin{equation*}
  |a_{Q_{i_1,\dots,i_s}}(x)|\le\frac{2^{R(i_1,\dots,i_s)+1}+2^{R(i_1,\dots,i_s)+2}}{\lambda_{Q_{i_1,\dots,i_s}}}\le\frac{1}{|Q_{i_1,\dots,i_s}|},
\end{equation*}
while for $x\in Q_{i_1,\dots,i_{s+1}}$
\begin{equation*}
  |a_{Q_{i_1,\dots,i_s}}(x)|\le\frac{|\langle f\rangle_{Q_{i_1,\dots,i_{s+1}}}-\langle f\rangle_{Q_{i_1,\dots,i_s}}|}{\lambda_{Q_{i_1,\dots,i_s}}}
  \le\frac{1}{|Q_{i_1,\dots,i_s}|}.
\end{equation*}
Thus, we obtain
\begin{equation*}
  \|a_{Q_{i_1,\dots,i_s}}\|_{L^1}\le1.
\end{equation*}

Observe
\begin{align*}
  \sum_{s=1}^\infty\sum_{i_1,\dots,i_s} |\lambda_{Q_{i_1,\dots,i_s}}|& = \sum_{s=1}^\infty\sum_{i_1,\dots,i_s}\big(2^{R(i_1,\dots,i_s)+1}+2^{R(i_1,\dots,i_s)+2}\big)|Q_{i_1,\dots,i_s}|+\\
   &\qquad+\sum_{s=1}^\infty\sum_{i_1,\dots,i_{s+1}}2^{R(i_1,\dots,i_{s+1})+1}|Q_{i_1,\dots,i_{s+1}}|\\
   &\le C\sum_{s=1}^\infty\sum_{i_1,\dots,i_{s}}2^{R(i_1,\dots,i_{s})+1}|Q_{i_1,\dots,i_{s}}|\\
   &=C\sum_{k=1}^\infty 2^k\sum_{\twoline{s=1,2,\dots}{i_1,\dots,i_s}\atop{R(i_1,\dots,i_s)}=k}|Q_{i_1,\dots,i_{s}}|\\
   &=(**)
\end{align*}
where the constant $C$ is absolute. We make the following 2 observations:
\begin{enumerate}[(a)]
\item for a fixed $R(i_1,\dots,i_s)=R(j_1,\dots,j_t)$ the cubes $Q_{i_1,\dots,i_s}$ and $Q_{j_1,\dots,j_t}$ are essentially disjoint. This follows, since if they weren't essentially disjoint, one would have to contain the other, which is impossible (unless, of course, they are the same cube).
\item we have
\begin{equation*}
  Q_{i_1,\dots,i_s}\subset\big\{x:M^*f(x)>2^{R(i_1,\dots,i_s)}\big\}
\end{equation*}
\end{enumerate}
so
\begin{equation*}
  (**)\le C\sum_{k=1}^\infty 2^k\big|\big\{x:M^*f(x)>2^k\big\}\big|\le C\|M^*f\|_{L^1},
\end{equation*}
with the constant $C$ absolute.
We will  further decompose the $a_{Q_{i_1,\dots,i_s}}$.

Let $Q_{i_1,\dots,i_{s+1}}^\#$ be the dyadic, immediate, parent of $Q_{i_1,\dots,i_{s+1}}$.
The cubes $Q_{i_1,\dots,i_{s+1}}^\#$ need not to be disjoint, so we fix the maximal one. By the construction
 it is contained in $Q_{i_1,\dots,i_s}$.
We denote by $Q_1,\dots,Q_n$ those cubes among the immediate dyadic descendants of
$Q_{i_1,\dots,i_{s+1}}^\#$ which belong to the set $\{Q_{i_1,\dots,i_{s+1}}\}_{s+1}$, while we denote by $Q_{n+1},\dots,Q_{2^d}$ the remaining descendants ($0<n\le 2^d$). Since the cube  $Q_{i_1,\dots,i_{s+1}}^\#$ is not one of the chosen cubes, thus
\begin{equation*}
  \big|\langle f\rangle_{Q_{i_1,\dots,i_{s+1}}^\#}\big|<2^{R(i_1,\dots,i_s)+2}.
\end{equation*}
Similarly,
\begin{align*}
  \big|\langle f\rangle_{Q_k}\big|<2^{R(i_1,\dots,i_s)+2}&\qquad k=n+1,\dots,2^d,\\
  \big|\langle f\rangle_{Q_k}\big|\ge2^{R(i_1,\dots,i_s)+2}&\qquad k=1,\dots,n.
\end{align*}
We have
\begin{align*}
  \sum_{k=1}^n|Q_k|\langle f\rangle_{Q_k} &=|Q_{i_1,\dots,i_{s+1}}^\#|\langle f\rangle_{Q_{i_1,\dots,i_{s+1}}^\#} -\sum_{k=n+1}^{2^d}|Q_k|\langle f\rangle_{Q_k}\\
  &\le|Q_{i_1,\dots,i_{s+1}}^\#|\cdot2^{R(i_1,\dots,i_s)+2}+\sum_{k=n+1}^{2^d}|Q_k|\cdot2^{R(i_1,\dots,i_s)+2}\\
  &\le|Q_{i_1,\dots,i_{s+1}}^\#|\cdot2^{R(i_1,\dots,i_s)+3}.
\end{align*}
Let us denote
\begin{equation*}
  C_k=a_{Q_{i_1,\dots,i_s}}\cdot\mathbbm{1}_{Q_k},\qquad k=1,\dots,n,
\end{equation*}
and we obtain
\begin{equation*}
  |\sum_{k=1}^n C_k|\le\frac{2^d\cdot2^{R(i_1,\dots,i_s)+3}}{\lambda_{Q_{i_1,\dots,i_s}}}\le\frac{2\cdot2^d}{|Q_{i_1,\dots,i_s}|}.
\end{equation*}
Let
\begin{equation*}
  \tilde C=\frac{1}{n}\,\sum_{k=1}^n C_k,
\end{equation*}
and let us adjust the value of the pre-atom $a_{Q_{i_1,\dots,i_s}}$ on cubes $Q_k$ from $C_k$ to $\tilde C$ ($k=1,\dots,n$). Call the adjusted pre-atom $\tilde a_{Q_{i_1,\dots,i_s}}$. Observe that the new functions have the same support, the same mean, and the $L^1$ norm is still $\le1$. Additionally, the new functions satisfy:
\begin{itemize}
\item outside $\bigcup_{i_{s+1}}Q_{i_1,\dots,i_{s+1}}$ we have
\begin{equation*}
  |\tilde a_{Q_{i_1,\dots,i_s}}(x)|\le\frac{1}{|Q_{i_1,\dots,i_s}|},
\end{equation*}
\item on $\bigcup_{i_{s+1}}Q_{i_1,\dots,i_{s+1}}$ we have
\begin{equation*}
  |\tilde a_{Q_{i_1,\dots,i_s}}(x)|\le\frac{2\cdot 2^d}{|Q_{i_1,\dots,i_s}|},
\end{equation*}
\item
\begin{equation*}
  \Big\{x:|\tilde a_{Q_{i_1,\dots,i_s}}(x)|\ge\frac{1}{|Q_{i_1,\dots,i_s}|}\Big\}\subset\bigcup_{i_{s+1}}Q_{i_1,\dots,i_{s+1}},
\end{equation*}
while each cube $Q_{i_1,\dots,i_{s+1}}$ is contained in appropriate parent $Q^\#= Q_{i_1,\dots,i_{s+1}}^\#$, for which we have
\begin{align*}
  &\frac{1}{|Q^\#|}\Big|\int_{Q^\#}\tilde a_{Q_{i_1,\dots,i_s}}(x)\,dx\Big|=\frac{1}{|Q^\#|}\Big|\sum_{k=1}^n\int_{Q_k}\tilde a_{Q_{i_1,\dots,i_s}}(x)\,dx+ \\
  &\qquad\qquad\qquad+\sum_{k=n+1}^{2^d}\int_{Q_k}\tilde a_{Q_{i_1,\dots,i_s}}(x)\,dx\Big|\\
  &\qquad\qquad=\frac{1}{|Q^\#|}\Big|\sum_{k=1}^n\tilde C+\sum_{k=n+1}^{2^d}\int_{Q_k}\tilde a_{Q_{i_1,\dots,i_s}}(x)\,dx\Big|\\
  &\qquad\qquad=\frac{1}{|Q^\#|}\Big|\int_{Q^\#} a_{Q_{i_1,\dots,i_s}}(x)\,dx\Big|.
\end{align*}
\end{itemize}
Denote by $b= b_{Q_{i_1,\dots,i_{s+1}}^\#}=a_{Q_{i_1,\dots,i_s}}-\tilde a_{Q_{i_1,\dots,i_s}}$.

Observe that $b\ne 0$
only on $S=\bigcup_{k=1}^{n}Q_k$.
By the construction we have $\|b\|_{L^1(S)}\le 2\|a_{Q_{i_1,\dots,i_s}} \|_{L^1(S)}$.
 Moreover, since  $b$ is constant on each $Q_k, 1\le k \le n$ it satisfies
 $|Q_{i_1,\dots,i_{s+1}}^\#|\|b\|_{L^\infty}\le 2^d \|b\|_{L^1}$
 and hence $\frac{b}{\|b\|_{L^1}}$ is an atom in the sense of Definition \ref{Def}.

Now we consider one of cubes $Q_{n+1},...,Q_{2^d}$, say $Q=Q_{n+1}$. Let again $Q^\#=Q_{i_1,\dots,i_{s+1}}^\#$
be  maximal contained in $Q$. We repeat the above procedure to $Q^\#$ and obtain new atom
$\tilde \tilde a_{Q_{i_1,\dots,i_s}}$ obtained by subsequent modification of $\tilde a_{Q_{i_1,\dots,i_s}}$
on subcubes of the new cube $Q^\#$ and atom $b_1= b_{Q_{i_1,\dots,i_{s+1}}^\#}$. Observe that the supports of $b,b_1$
are disjoint. We continue recurrently packing up the cubes $Q^\#$. As we finish, we have the atom $\tilde{\tilde {\tilde a}}$ (multiple tildas) satisfying conditions of Definition \ref{Def} (verification of that is immediate once we observe that $a$ has been modified only on cubes $Q_s$ for which we use Definition \ref{Def} (d)), and a sequence of correction atoms $b_j$'s of disjoint supports. We thus have
$\sum_j \|b_j\|_{L^1}\le  2\|a_{Q_{i_1,\dots,i_s}} \|_{L^1}$. The theorem follows.
\end{proof}
{\bf Example.}
Let $\varphi$ be a radial kernel, with support in ball of radius $1+\frac1{d}$, with its radial profile constant on ball of radius 1, linear for $1\le |x|\le 1+\frac1{d}$, such that $\int \varphi=1$. Let $\varphi_t$ be the $L^1$ normalized dilation,
$\varphi_t(x)=\frac1{t^d}\varphi(\frac{x}{t})$. We will prove the following
\begin{theorem}\label{theorem}
 For an atom $a$ satisfying axioms of Definition \ref{Def}, supported on the cube $[0,1]^d$, we have
\begin{equation}\label{MM_est}
\|\sup_{t\le\frac1{d}}\varphi_t*a\|_{L^1}\le Cd\log(d)\|a\|_{L^1}.
\end{equation}
with $C$ an absolute constant. As a consequence, the operator norm of the the maximal function
\begin{equation}\label{ME_est}
Mf(x)=|\sup_{t>0}\varphi_t*f(x)|
\end{equation}
acting from $H^1_d\rightarrow L^1$, is at most $Cd\log(d)$.
\end{theorem}
We also prove the lower estimate $C\frac{d}{\log(d)}$, see comments and the end of the proof.
\begin{proof}
Let us fix an atom $\a$ supported on $\mathbf Q=[0,1]^d$.
We begin with a sequence of lemmas.
\begin{lemma} {\label{prep_lem_1}}
Let $Q$ be a cube with sidelength $\rho$, $y\in Q$, $t\ge d^\frac32\rho$ and
\begin{equation*}
  \mathbbm{1}_t=\mathbbm{1}_B, \qquad\text{where }B=B\big(0,t(1+\textstyle\frac{2}{d})\big),
\end{equation*}
(ball of center 0 and radius $t(1+\frac{1}{d}))$. We then have
\begin{equation*}
  |\varphi_t(x-y_c)-\varphi_t(x-y)|\ \le\ C\cdot\frac{d}{t}\cdot\mathbbm{1}_t(x-y)\cdot\frac{1}{|B|}\cdot\|y\|,
\end{equation*}
where $y_c$ is the center of $Q$. Consequently, for an atom $a$ supported on a cube $Q$, satisfying $\|a\|_{L^\infty}\le 1$, with $Q$ and $t$ as above, we have
\begin{equation*}
|\varphi_t*a(0)|\ \le\ C\cdot\frac{d\sqrt{d}\rho}{t}\cdot\|\mathbbm{1}_t\|_{L^1(Q)}.
\end{equation*}
\end{lemma}
As a corollary to Lemma \ref{prep_lem_1} we immediately have
\begin{corollary}\label{coro_infty}
Suppose $a_1,a_2,\dots$ are atoms supported on disjoint cubes of sidelengths $\rho$, all contained in some cube $Q$.
Assume $\|a_i\|_{L^\infty}\le 1$ and $t\ge d^\frac32\rho$. Then
\begin{equation}\label{Jel_inf}
  \Big|\varphi_t*\sum_i a_i(0)\Big| \le\ C\cdot\frac{d^{3/2}\rho}{t}\cdot\|\mathbbm{1}_t\|_{L^1(Q)}.
\end{equation}
\end{corollary}

We recall that according to Definition \ref{Def} (d) for each atom $a$ we have distinguished cubes $Q_s$ such that $a$ is constant (with value no greater than $2^d/|Q_s|$) on each of the $Q_s$'s. We will call these distinguished cubes black. For a black cube $Q$ the value of $a$ on $Q$ will be denoted $\alpha_Q$.

\begin{lemma}\label{at_decomp_1}
 Let us fix an integer $s$ with $2^{-s}\simeq t d^{-\frac32}$. Let $\mathcal M$ be the family of the maximal $Q^\#$ with sidelength $\le 2^{-s}$ (cubes $Q^\#$ are the parent cubes of black cubes given by Definition \ref{Def} (d) for the fixed atom $\a$). The atom $\a$, decomposes as a sum
\begin{equation*}
  \a=a_1^s+a_2^s+a_3^s,
\end{equation*}
where
\begin{equation*}
  a_1^s(x)=\begin{cases}\hfil 0&:\text{ on any black cube with sidelength $\ge2^{-s}$},\\\frac{\mathbbm{1}_{Q^\#}(x)}{|Q^\#|}\int_{Q^\#}\a&:\ x\in Q^\#,Q^\#\in\mathcal M,\\
  \hfil a(x)&:\text{ otherwise},\end{cases}
\end{equation*}
\begin{equation*}
  a_2^s(x)=\sum_i \mathbbm{1}_{Q_i}(x) \a(x),
\end{equation*}
where $Q_i$'s are all the black cubes with sidelengths $\ge 2^{-s}$, and
\begin{equation*}
  a_3^s(x) =\begin{cases} \a(x)-\frac{\mathbbm{1}_{Q^\#}(x)}{|Q^\#|}\int_{Q^\#}\a&:\ x\in Q^\#,Q^\#\in\mathcal M,\\
  \hfil 0&:\text{ otherwise}.\end{cases}
\end{equation*}
Clearly we have
\begin{equation*}
  |\varphi_t*a_1^s|\ \le\ C,
\end{equation*}
and, moreover, for $t\le\frac{1}{d}$
\begin{equation*}
\text{\rm{supp}}\,(\varphi_t*a_1^s)\subset(1+\textstyle\frac{4}{d})\mathbf{Q}.
\end{equation*}
\end{lemma}
(For a cube $Q$ and a positive number $s$ $sQ$ means cube with the same center as $Q$, and sidelength $s$ times the sidelength of $Q$.)
\begin{corollary}
We have
\begin{align*}
  |\varphi_t*\a(x)|&\le C\cdot\mathbbm{1}_{(1+\frac{4}{d})\mathbf{Q}}(x)+
  \sum_{\twoline{\text{$Q_i$ black,}}{l(Q_i)\ge d^\frac52 2^{-s}}}
  |\alpha_{Q_i}|\mathbbm{1}_{(1+4/d)Q_i}(x)+\\
  &\quad+\sum_{\twoline{\text{$Q_i$ black,}}{2^{-s}\le l(Q_i)<d^\frac52 2^{-s}}}
  |\alpha_{Q_i}|\cdot\varphi_t*1_{Q_i}(x)+|\varphi_t*a_3^s(x)|,\\
  &=I+II+III+IV
\end{align*}
where we denote by $l(Q)$ the sidelength of $Q$.
\end{corollary}
First two summands give rise to the $L^1$ control of the maximal function with constants independent of the dimension.
For the third summand we have uniform in $t$ estimate
\begin{equation*}
 III\le \sum_{\text{$Q_i$ black,}}
  |\alpha_{Q_i}|\sup_{{d^{-1} l(Q_i)\le t\le l(Q_i)d^\frac32}}\varphi_t*\mathbbm{1}_{Q_i}(x)
\end{equation*}
and since
\begin{equation*}
\int_{{d^{-1} l(Q_i)\le t\le l(Q_i)d^\frac32}}\|\partial_t\varphi_t\|_{L^1}\le Cd\log(d)
\end{equation*}
and
\begin{align}\label{J1_est}
J&=\sup_{d^{-1} l(Q_i)\le t\le l(Q_i)d^\frac32}\varphi_t*\mathbbm{1}_{Q_i}(x)\notag\\
&\le\int_{{d^{-1} l(Q_i)\le t\le l(Q_i)d^\frac32}}|\partial_t\varphi_t|*\mathbbm{1}_{Q_i}(x)dt
\end{align}
we have $\|III\|_{L^1}\le Cd\log(d)\|a\|_{L^1}$.

The last summand IV will be estimated by Corollary \ref{coro_infty}.

\begin{lemma} We have
\begin{align*}
  a_3^s(x) &=\sum_{n\ge s}\big(E_{-n} \a(x)-E_{-n-1}\a(x)\big)\\
   &=\sum_{n\ge s}\sum_{Q\in\mathcal D_n}\Big(\frac{\mathbbm{1}_Q}{|Q|}\int_Q\a-\sum_{\twoline{Q'\subset Q}{Q'\in\mathcal C(Q)}}\frac{\mathbbm{1}_{Q'}}{|Q'|}\int_{Q'}\a\Big)\\
   &=\sum_{n\ge s}\sum_{\twoline{Q\in\mathcal D_n}{\text{all }Q'\in\mathcal C(Q),\text{ not black}}} \ +\ \text{remainder}\\
   &=I\ +\ II,
\end{align*}
where $\mathcal D_n$ denotes the family of dyadic cubes of sidelengths $2^{-n}$, and $\mathcal C(Q)$ denotes the family of immediate dyadic descendants of a cube $Q$.
\end{lemma}
Observe, that for a fixed $n\ge s$, $Q\in \mathcal D_n$, $Q$ of type contained in $I$, the $a_Q=a_3^s\cdot\mathbbm{1}_Q$
\begin{equation*}
a_Q(x)=\frac{\mathbbm{1}_Q(x)}{|Q|}\int_Q\a-\sum_{\twoline{Q'\subset Q}{Q'\in\mathcal C(Q)}}\frac{\mathbbm{1}_{Q'}(x)}{|Q'|}\int_{Q'}\a
\end{equation*}
satisfy the assumptions of Corollary \ref{coro_infty} with $\rho= 2^{-n}=2^{-s}2^{-l}$, $2^{-s}d^{\frac32}\simeq t$.
Hence, by \eqref{Jel_inf} we obtain $|\sum_{Q\in \mathcal D_n}\varphi_t*a_Q|\le C2^{-l}$, and we can sum up with respect to $l$.
As a result, we  obtain a dimension free $L^\infty$ bound.
Let $Q\in \mathcal D_n,\ n\ge s$ be such that at least one $Q'\in\mathcal C(Q)$ is black. We will then say that $Q$ has type 2. We decompose $a_Q$ further into average $0$ functions
\begin{equation*}
a_Q(x)=b_Q(x)+e_Q(x)
\end{equation*}
where
\begin{equation*}
b_Q(x)=\frac{\mathbbm{1}_Q(x)}{|Q|}\int_{\cup Q':Q'\text{ black}}a-\sum_{Q'\in\mathcal C(Q)\text{ black}}\frac{\mathbbm{1}_{Q'}(x)}{|Q'|}\int_{Q'}a
\end{equation*}
and
\begin{equation*}
e_Q(x)=\frac{\mathbbm{1}_Q(x)}{|Q|}\int_{\cup Q':Q'\text{ not black}}a-\sum_{Q'\in\mathcal C(Q)\text{ not black}}\frac{\mathbbm{1}_{Q'}(x)}{|Q'|}\int_{Q'}a
\end{equation*}
Observe, that the family $e_Q(x)$ satisfies again the condition of the Corollary \ref{coro_infty}, so  by the preceding case
argument we get
\begin{equation*}
\Big|\sum_{Q\in \mathcal D_n,Q\text{ of type 2 }}\varphi_t*e_Q\Big|\le C\cdot2^{-l}
\end{equation*}
and we again sum up to obtain a dimension free $L^\infty$ bound.

We are left with the estimate for
\begin{equation*}
J=\sum_{n\ge s}\Big|\sum_{Q\in \mathcal D_n,Q\text{ of type 2 }}\varphi_t*b_Q(x)\Big|
\end{equation*}
where we have an additional relation $2^{-s}d^{\frac32}\simeq t$. We have
\begin{equation*}
J\le \sum_{Q\text{ of type 2 }}|\sup_{t\ge d^{\frac32}l(Q)}\varphi_t*b_Q(x)|
\end{equation*}
and the right hand side does not depend on $t$. Observe, that by standard cancellation argument
\begin{equation}\label{J2_est}
\Big\|\sup_{t\ge d^{\frac32}l(Q)}\varphi_t*b_Q\Big\|_{L^1}\le Cdiam(Q)\|b_Q\|_{L^1}
\sup_{t\ge d^{\frac32}l(Q)}\frac{d\cdot\mathbbm{1}_{B(0,t(1+2/d))}(x)}{t^{d+1}|B(0,1+2/d)|}
\end{equation}
We have, for $|x|\ge(1+1/d)d^{\frac32}d(Q)$
\begin{equation*}
\sup_{t\ge d^{\frac32}l(Q)}\frac{d\cdot\mathbbm{1}_{B(0,t(1+2/d))}(x)}{t^{d+1}|B(0,1+2/d)|}=
\frac{d(1+2/d)^d}{|B(0,1+2/d)||x|^{d+1}}
\end{equation*}
and integrating in polar coordinates, the expression \eqref{J2_est} has $L^1$ norm bounded by $Cd$.
Since the case $|x|\le(1+\frac1{d})d^{\frac32}l(Q)$ is immediate, the main estimate \eqref{MM_est} follows.

The estimates of the maximal functiom over the intervals $\frac1{d}\le t \le d^\frac32$ and $t\ge d^\frac32$ follows similarly
to \eqref{J1_est}, \eqref{J2_est}. We leave the details for the reader. Theorem \ref{theorem} follows.
\end{proof}

We now briefly sketch the argument leading to the maximal function estimates from below.
We recall  that $B,|B|$ denote the unit ball in $\er^d$ and its Lebesgue measure.

First observe, that for $A=2^{2[\log(d)]}\approx d^2$, the function
\begin{equation}
h(x)=2^{-d}\mathbf{1}_{[-1,1]^d}(x)-(2A)^{-d}\mathbf{1}_{[-A,A]^d}(x)=h_1(x)-h_2(x)
\end{equation}
defined on $\er^d$ has $\har$ norm of order $\log(d)$. This can be easily checked by the formula
\begin{equation}
h(x)=\sum_{s=0}^{2[\log(d)]-1}2^{-ds}\mathbf{1}_{[-2^s,2^s]^d}(x)-  2^{-d(s+1)}\mathbf{1}_{[-2^{s+1},2^{s+1}]^d}(x)    =
   \sum_{s=0}^{2[\log(d)]-1}h^s(x)
\end{equation}
It can be easily checked, that the expectation of each  $h^s$ over a grid of the dyadic cubes of sidelength $2^{l}$, $l\ge s+1$
vanish, and that the expectation over a dyadic cubes of sidelength $2^{l}$, $l\le s$ leaves $h^s$ unchanged.
Consequently $h^s$ has its $\har$ norm equal to $2$.

Then we consider linearized maximal operator $Th(x)=\varphi_{t(x)}*h(x)$, where we will assume
$t(x)=|x|+4\le 3d$ for  $|x|\le 2d$. Then observe that $Th_2(x)=(2A)^{-d}$ for $|x|\le Cd$ and this function restricted
to the ball of radius $2d$ is of the $L^1$ norm of order $O(1)$ (and even smaller).

Now observe, that  the $L^1$ norm $Th_1(x)$ restricted to the ring $d\le |x|\le 2d$ is at least
$cd$, where $c>0$ is dimension free. The crucial observation is that if $t(x)=|x|+4$, the ball $B(t)$
covers all points in the support of $h_1$ lying below (that is in the direction of $x$) the
hyperplane passing through $0$ and perpendicular to $x$. Since $\varphi(x)\ge
\frac{c_0}{|B|}\mathbf{1}_{B}$, where $c_0$ is a dimension free constant, as a result we have  $Th_1(x)\ge
\frac{c_0}{2|B|}(|x|+4)^{-d}$. and the statement follows by integration in polar coordinates.
\begin{remark}\label{rota} The above argument applied to the classical heat maximal function leads to the
corresponding lower and upper bounds: $\frac{Cd^\frac12}{\log(d)}$, $Cd^\frac12\log(d)$.
We note, that \lq\lq near $L^1$" approach based on Rota's theorem seems to give upper estimate equal to $Cd$.
\end{remark}
\begin{remark} The following, easy to prove, inequality is very useful in obtaining the estimates from below
\begin{equation}\label{rad_low}
Ma(x)\ge Mr(a)(x)
\end{equation}
where $M$ denotes the maximal function with respect to a radial  kernel, and $r(a)$ is the radialisation of a function $a$. Using \eqref{rad_low} one can obtain lower bound $Cd$ for an example considered in Theorem \ref{theorem}.
\end{remark}
\begin{remark}\label{rmrk} The following lemma can be used to obtain an $L^2$ version of our argument.
\begin{lemma*} Let
\begin{equation*}
T(f)=\sum_{Q^\#}\sum_{Q_c\sim Q^\#}\frac{\mathbf{1}_{Q^\#}}{|Q^\#|}\int_{Q_c}f,
\end{equation*}
where $Q_c$ are the black cubes, whose immediate dyadic parent is $Q^\#$. Then
\begin{equation*}
  \int_{Q}|T(f)|^2\le A\int_Q|f_c|,
\end{equation*}
where $f_c$ is the restriction of $f$ to the black cubes, and $A$ is a universal constant.
\end{lemma*}
\begin{proof} We have
\begin{equation*}
  \int_{Q}|T(f)|^2=\int_Q\sum_{Q_1^\#,Q_2^\#\subset Q}\sum_{\twoline{Q_c^1\sim Q_1^\#}{Q_c^2\sim Q_2^\#}}
  \frac{\mathbbm{1}_{Q_1^\#}\mathbbm{1}_{Q_2^\#}}{|Q_1^\#||Q_2^\#|}\int_{Q_c^1}f\int_{Q_c^2}\overline{f}=(*).
\end{equation*}
Clearly, only intersecting $Q_1^\#,Q_2^\#$ need to be considered, and there are two possibilities:  $Q_1^\#\subseteq Q_2^\#$ or $Q_2^\#\subset Q_1^\#$. We proceed with the first case, the second is similar.
\begin{align*}
  (*) & = \sum_{Q_1^\#\subseteq Q_2^\#\subset Q}\sum_{\twoline{Q_c^1\sim Q_1^\#}{Q_c^2\sim Q_2^\#}}
  \frac{1}{|Q_2^\#|}\int_{Q_c^1}f\int_{Q_c^2}\overline{f} \\
   & = \sum_{Q_2^\#\subset Q}\frac{1}{|Q_2^\#|}\int_{\twoline{\bigcup Q_c - \text{black}}{Q_c\sim Q_2^\#}}f\sum_{Q_1^\#\subseteq Q_2^\#}\int_{\twoline{\bigcup Q_c - \text{black}}{Q_c\sim Q_1^\#}}\overline{f}
\end{align*}
All sums and unions are over essentially disjoint cubes. In particular
\begin{equation*}
  \sum_{Q_1^\#\subseteq Q_2^\#}\int_{\twoline{\bigcup Q_c - \text{black}}{Q_c\sim Q_1^\#}}\overline{f}=\int_{\twoline{\bigcup Q_c-\text{black}}{Q_c\subset Q_2^\#}}\overline{f}.
\end{equation*}
Since
\begin{equation*}
  \Big|\int_{Q_2^\#}\overline{f}\Big|\le C|Q_2^\#|,\qquad\text{$C$ - universal},
\end{equation*}
by the definition of the ``parent'' cubes, and
\begin{equation*}
  \Big|\int_{\twoline{Q_2^\#\setminus\bigcup Q_c}{Q_c-\text{black}}}\overline{f}\Big|\le C|Q_2^\#|,\qquad\text{$C$ - universal},
\end{equation*}
by the definition of the black cubes, we also have
\begin{equation*}
  \Big|\int_{\twoline{\bigcup Q_c-\text{black}}{Q_c\subset Q_2^\#}}\overline{f}\Big|\le C|Q_2^\#|,\qquad\text{$C$ - universal}.
\end{equation*}
This concludes the proof:
\begin{align*}
  (*)&=|(*)|\\
  &\le \sum_{Q_2^\#\subset Q}\frac{1}{|Q_2^\#|}\Big|\int_{\twoline{\bigcup Q_c - \text{black}}{Q_c\sim Q_2^\#}}f \Big|C|Q_2^\#|\\
  &\le C\sum_{Q_2^\#\subset Q}\int_{\twoline{\bigcup Q_c - \text{black}}{Q_c\sim Q_2^\#}}|f|\\
  &=C\int_Q|f_c|.
\end{align*}
\end{proof}
\end{remark}

\end{document}